\documentclass[11pt,letterpaper,upref]{amsart}
\usepackage[letterpaper,margin=1.5in]{geometry}
\usepackage{amssymb}
\usepackage{amsxtra}
\usepackage[mathscr]{eucal}
\usepackage{graphics}
\usepackage{mathtools}
\usepackage[nobysame]{amsrefs}
\usepackage{color}
\usepackage{enumitem} 

\numberwithin{equation}{section}

\newtheorem{theorem}{Theorem}[section]
\newtheorem{proposition}[theorem]{Proposition}
\newtheorem{lemma}[theorem]{Lemma}

\newtheorem{conjecture}[theorem]{Conjecture}

\theoremstyle{definition}

\newcommand{\af}{\alpha_f}

\newcommand{\al}{\alpha}

\newcommand{\aut}{\operatorname{aut}}

\newcommand{\csp}{\hspace{0.05em}}

\newcommand{\Df}{\Delta_{f}}
\newcommand{\Del}{\Delta}
\newcommand{\del}{\delta}
\newcommand{\xxx}{E_{T\!,\,\om}}

\newcommand{\lifr}{\ell^\infty(\FF,\RR)}
\newcommand{\lifz}{\ell^\infty(\FF,\ZZ)}

\newcommand{\FF}{\mathbb{F}}
\newcommand{\fhat}{\widehat{f}}

\newcommand{\G}{\Gamma}

\newcommand{\F}{\mathbb{F}}
\newcommand{\fs}{f^*}

\newcommand{\mf}{\mu_f}
\newcommand{\muhat}{\widehat{\mu}}
\newcommand{\nuhat}{\widehat{\nu}}
\newcommand{\om}{\omega}

\newcommand{\phs}{\phi_*}
\newcommand{\pt}{\partial T}

\newcommand{\RR}{\mathbb R}
\newcommand{\sig}{\sigma}
\newcommand{\Sig}{\Sigma}
 
\newcommand{\tb}{\overline{T}}
\newcommand{\TT}{\mathbb{T}}
\newcommand{\tf}{\TT^{\FF}}
\newcommand{\wdel}{w^{\Delta}}
\newcommand{\xdel}{x^{\Delta}}
\newcommand{\xf}{X_f}
\newcommand{\ZG}{\ZZ\G}
\newcommand{\ZZ}{\mathbb{Z}} 
\newcommand{\zd}{\ZZ^{d}}

\renewcommand{\ge}{\geqslant}
\renewcommand{\le}{\leqslant}
\newcommand{\<}{\langle}
\renewcommand{\>}{\rangle}  
\renewcommand{\emptyset}{\varnothing}

\begin{document}

\title{New Examples of Bernoulli Algebraic Actions}

\author{Douglas Lind}

\address{Douglas Lind: Department of Mathematics, University of
  Washington, Seattle, Washington 98195, USA}
  \email{lind@math.washington.edu}

\author{Klaus Schmidt}

\address{Klaus Schmidt: Mathematics Institute, University of Vienna, 
Nordberg\-stra{\ss}e 15, A-1090 Vienna, Austria}
\email{klaus.schmidt@univie.ac.at}





\begin{abstract}
   We give an example of a principal algebraic action of the noncommutative free group $\FF$ of rank two by automorphisms of a connected compact abelian group for which there is an explicit measurable isomorphism with the full Bernoulli 3-shift action of $\FF$. The isomorphism is defined using homoclinic points, a method that has been used earlier to construct symbolic covers of algebraic actions. To our knowledge, this is the first example of a Bernoulli algebraic action of $\FF$ without an obvious independent generator. Our methods can be generalized to a large class of acting groups. 
\end{abstract}

\maketitle

\section{Introduction}\label{sec:introduction} Halmos \cite{Halmos} first observed that an continuous automorphism of a compact group automatically preserves Haar measure, providing a rich class of examples in ergodic theory. Using Pontryagin duality theory, it is possible to obtain explicit and concrete answers to dynamical questions. In particular, a series of papers in the 1970s culminated in the definitive result that every ergodic automorphism of a compact abelian group is measurably isomorphic to a Bernoulli shift \cites{LindStructure, Miles-Thomas}.

The study of the joint action of several commuting automorphisms of a compact abelian group was initiated by Bruce Kitchens and the second author \cite{KitSch}. This has ultimately led to a detailed understanding of such actions, called algebraic $\zd$-actions, as described in \cite{DSAO}. Here there is a natural necessary condition for such actions to be measurably isomorphic to Bernoulli shifts, namely having completely positive entropy, and this condition can be checked using commutative algebra \cite{LSW}*{Thm.\ 6.5}. The second author and Dan Rudolph showed in \cite{RudolphSchmidt} that this condition is also sufficient. 

For acting groups that are not commutative, much less is known.  See for example our recent survey \cite{LS-Heis} of algebraic actions of the discrete Heisenberg group. Even for this group we do not know a general method to decide whether or not a given algebraic action is measurably isomorphic to a Bernoulli action.

The study of actions of general countable groups, even beyond amenable groups, has been revolutionized by Lewis Bowen's introduction of new ideas about entropy and independence. The recent book of Kerr and Li \cite{KerrLi} gives a comprehensive account of these developments, in particular of how entropy can be defined for actions of sofic groups. Algebraic actions supply a large class of interesting examples for this theory. The study of entropy for algebraic actions of noncommutative groups was initiated by Christopher Deninger \cite{Deninger}, who showed that entropy could be computed for many amenable groups using the Fuglede-Kadison determinant of an associated operator in a von Neumann algebra. This insight was developed in a series of papers by several authors, leading to a definitive form for algebraic actions of general sofic groups by Hayes \cite{Hayes}.

However, little is known about when algebraic actions of sofic groups are measurably isomorphic to Bernoulli actions. The reason for this ignorance is that many of the essential results for the Bernoulli theory of amenable group actions due to Ornstein and Weiss \cite{Ornstein-Weiss} fail for sofic groups. For instance, factors of Bernoulli actions may fail to be Bernoulli. A striking example of this is due to Popa \cites{Popa, Popa-Sasyk}: the algebraic action of a countable group $\G$ having property (T) on the quotient of ~$\TT^\G$ by the subgroup of constant points is not Bernoulli (see \cite{Bowen}*{Thm.\ 7.2} for a succinct explanation).

In this paper we construct an explicit measurable isomorphism between an algebraic action of the (noncommutative) free group $\FF$ of rank two on a connected compact abelian group and the full 3-shift action of $\FF$ which preserves the respective measures. We believe that this is the first nontrivial example of this sort, where there is no obvious independent generator.  Our proof uses symbolic covers, homoclinic points, and a percolation argument from \cite{EinsiedlerSchmidt}*{Prop.\ 5.1}. That argument relies on the fact that for expansive algebraic $\ZZ^d$-actions Haar measure is the unique measure of maximal entropy. However it an open question whether this remains true for ~$\FF$. Here we give an alternative argument, showing that the image of the 3-shift measure is invariant under translations by all elements in the dense homoclinic group,  and hence it must be Haar measure.

Our methods can be generalized to the class of so-called indicable groups, namely those groups for which there is a surjective homomorphism to $\ZZ$. In this setting, recent work of Hayes \cite{HayesHarmonic} provides a systematic way for proving an image measure is Haar using Fourier coefficients. We use this alternative to the homoclinic group argument while extending our results to indicable groups and other algebraic actions. By a result of David Kerr \cite{KerrCompletelyPositive}, these algebraic actions have completely positive entropy.

\section{Algebraic actions}\label{sec:algebraic-actions}

Let $\G$ be a countable discrete group with identity element $1_\G$.  An \emph{algebraic $\G$-action} on a compact abelian group $X$ is a homomorphism $\al\colon\G\to\aut(X)$ from $\G$ to the group of (continuous) algebraic automorphisms of $X$.  We denote the image of $s\in\G$ under $\al$ by $\al^{s}$, so that $\al^{st}=\al^{s}\circ\al^{t}$ and $\al^{1_\G}=\text{Id}_{X}$. 

Let $\ZZ\G$ denote the integral group ring of $\G$, consisting of all sums of the form $g=\sum_{s\in\G}g_s s$, where $g_s\in\ZZ$ for every $s\in\G$ and only finitely many $g_s$ are nonzero. The (additive) Pontryagin dual of $\ZG$ is $\TT^{\G}$, where $\TT=\RR/\ZZ$ and the dual pairing is given by $\<x,g\>=\sum_{s\in\G} x_sg_s$ for  $x\in\TT^{\G}$ and $g\in\ZG$. Left multiplication by $\G$ on $\ZG$ defines a $\G$-action that dualizes to the algebraic $\G$-action $\sigma$ on $\TT^{\G}$ given by $\sigma^s(x)_t^{}=x_{s^{-1}t}$.

Fix an $f\in\ZZ\G$, and let $\ZZ\G f$ denote the left principal ideal in $\ZZ\G$ generated by ~$f$. The compact dual group of $\ZZ\G/\ZZ\G f$ is then a subgroup of $\TT^{\G}$ denoted by $X_f$, and the restriction of $\sigma$ to $X_f$ is an algebraic $\G$-action denoted by $\af$. We call $(X_f,\af)$ the \emph{principal algebraic $\G$-action defined by $f$}. This action automatically preserves Haar measure $\mu_f$ on $X_f$.

A convenient concrete description of principal actions uses formal sums. Identify $x\in\TT^\G$ with the sum $\sum_{t\in\G}x_t\csp t$, where $x_t\in\TT$ for every $t\in\G$. Then $\G$ acts on $\TT^\G$ by left multiplication.  Explicitly,
\begin{displaymath}
   \sigma^s(x) = s\cdot x = s\cdot \sum_{t\in\G}x_t\csp t = \sum_{t\in\G} x_t\csp st=\sum_{t\in\G} x_{s^{-1}t}\csp t,
\end{displaymath}
so that $\sigma^s(x)_t=x_{s^{-1}t}$. Similarly, if $g\in\ZZ\G$ we can formally multiply $x$ by $g$ on the right by expanding out and collecting terms, with the result denoted by $x\cdot g$. We can express the dual pairing by $\<x,g\>=(x\cdot g^*)_{1_\G}$. Let $f^* = \sum_{s\in\G} f_s\csp s^{-1}$. Then $X_f$ is the subgroup of $\TT^\G$ consisting of all $x$ for which $x\cdot f^*=0$. Thus $x\in X_f$ if and only if $\sum_{s\in\G}f_s \csp x_{ts}=0$ for every $t\in\G$, a finite integral condition on the coordinates of $x$. 

We will use a similar convention for other spaces as well, for instance $\ell^\infty(\G,\RR)$ and $\ell^\infty(\G,\ZZ)$.

\section{The homoclinic map}\label{sec:homoclinic-map}

Our focus will be on algebraic actions of the free group $\FF$ of rank two generated by $a$ and $b$. We let $S=\{a,b,a^{-1},b^{-1}\}$ be the standard generating set, and use $S$ to define the word metric $|\cdot|_S$ on $\FF$. Our main example is the principal algebraic $\FF$-action defined by $f = 3 -a-b\in\ZZ\FF$.

First observe that $f^*=3-a^{-1}-b^{-1}$ is invertible in $\ell^1(\FF,\RR)$ by using geometric series. Specifically, if $N$ denotes the set of all words in $a^{-1}$ and $b^{-1}$ (including $1_\FF$), then $(f^*)^{-1}=(1/3)\sum_{s\in N} 3^{-|s|}\csp s$, which we denote by $\wdel$. 

Note that $0\le \wdel_s\le 1/3$ for every $s\in\FF$. Hence by putting $a=b=1$, we see that 
\begin{displaymath}
   1=f^*(1,1)^{-1}=\sum_{s\in\FF}\wdel_s=\|\wdel\|_1,
\end{displaymath}
and hence $\wdel$ is a probability distribution on $\FF$. For every $d\in\lifz$ we define $(d\cdot\wdel)_s = \sum_{t\in\FF} d_{t^{}}^{}\csp w_{t^{-1}s}^{\Delta}$. Clearly $|(d \cdot \wdel)_s|\le \|d\|_\infty \|\wdel\|_1$ for every $s\in\FF$. Hence we can define $\Phi\colon\lifz\to\lifr$ by $\Phi(d)=d\cdot\wdel$.

Let $\pi\colon\lifr\to\tf$ be the projection map defined by reducing each coordinate (mod 1). Clearly $\pi$ is continuous and equivariant. The composition $\phi=\pi\circ\Phi\colon\lifz\to\tf$ is called the \emph{homoclinic map}. Since $\wdel\cdot f^*=1_\G$, if $d\in\lifz$ then
\begin{displaymath}
   \phi(d)\cdot f^* = \pi\bigl(\Phi(d)\bigr)\cdot f^* =\pi(d\cdot\wdel)\cdot f^*
   =\pi(d\cdot\wdel\cdot f^*)=\pi(d)=0,
\end{displaymath}
and hence the image of $\phi$ is contained in $X_f$.

A point $x\in X_f$ is \emph{homoclinic} if $\lim_{|s|\to\infty}x_s=0$. The subset $\Delta_f$ of all homoclinic points is clearly a subgroup of $X_f$, called the \emph{homoclinic group of $\af$}. 

Considering $1_\FF\in\ZZ\FF$ as an element of $\lifz$, we put $\xdel=\phi(1_\G)=\pi(\wdel)$. Since $\FF$ is residually finite, the results of \cite{DeningerSchmidt}*{\S4} apply to show that $\af$ is expansive, that $\Del_f=\phi(\ZZ\FF)$ consists of all finite integral combinations of shifts of $\xdel$, and that $\Del_f$ is dense in $X_f$. This density plays a key role in \S\ref{sec:isomorphism}.

Let $Y=\{0,1,2\}^\FF\subset\ell^\infty(\FF,\ZZ)$, and let $\nu$ denote product measure on $Y$ with each symbol having measure $1/3$. Then the standard shift-action $\sigma$ of $\FF$ on $Y$ preserves $\nu$, and is called the \emph{full 3-shift action of $\FF$}.

\begin{theorem}\label{thm:main}
   Let $(Y,\sigma,\nu)$ be the full 3-shift action of $\FF$ and $(X_f,\af,\mu_f)$ be the principal algebraic $\FF$-action defined by $f=3-a-b$. The homoclinic map $\phi\colon Y \to X_f$ given by $\phi(d)=\pi(d\cdot\wdel)=d\cdot \xdel$ is continuous, equivariant, surjective, one-to-one off a $\nu$-null set, and $\phi_*\nu=\mu_f$. Thus $\phi$ is a measurable isomorphism between $(X_f,\af,\mu_f)$ and a Bernoulli shift.
\end{theorem}

\section{Symbolic covers}\label{sec:symbolic-covers}

In this section we find bounded subsets of $\lifz$ that are mapped onto $X_f$ by the homoclinic map $\phi$, i.e., symbolic covers of $(X_f,\af)$.

\begin{lemma}\label{lem:4-cover}
   $\phi\bigl(\{0,1,2,3\}^\FF\bigr)=X_f$.
\end{lemma}

\begin{proof}
   Let $x\in X_f$. There is a unique $v\in [0,1)^\FF$ with $\pi(v)=x$. Since
   \begin{displaymath}
      0_{X_f}=x\cdot f^*=\pi(v)\cdot f^*=\pi(v\cdot f^*),
   \end{displaymath}
   it follows that $v\cdot f^*\in\ell^\infty(\FF,\ZZ)$. Simple inequalities imply that $-1\le v\cdot f^*\le2$ coordinate-wise. Let $\mathbf{1}=\sum_{s\in\FF}s$, so that $\Phi(\mathbf{1})=\mathbf{1}\cdot\wdel=\mathbf{1}$ since $\wdel$ is a probability distribution. Then $d := v\cdot f^*+\mathbf{1}\in\{0,1,2,3\}^\FF$, and since $f^*\cdot\wdel=1$, we have that $\phi(d)=\pi\bigl((v\cdot f^* +\mathbf{1})\cdot\wdel\bigr)=\pi(v+\mathbf{1})=\pi(v)=x$.
\end{proof}

Let $C=[\FF,\FF]$ be the commutator subgroup of $\FF$, so that $\FF/C \cong \ZZ^2$ with commuting generators $aC$ and $bC$. Define a homomorphism $\FF/C\to\ZZ$ by mapping both $aC$ and $bC$ to $1\in\ZZ$. For $s\in\FF$ let $[ s ]$ denote the image of $sC$ in $\ZZ$. Then $[\, \cdot\, ]:\FF\to\ZZ$ is a surjective homomorphism with $[sa]=[sb]=[s]+1$ for every $s\in\FF$. For example, $[a^2 b^{-3} a^{-1}b] = 2 - 3 -1 + 1= -1$. Clearly $\bigl|[s]\bigr|\le |s|_S$ for all $s\in\FF$.

We will use  $[\, \cdot\,] $ to improve the previous result to obtain an optimal symbolic cover of $X_f$.

\begin{lemma}\label{lem:3-cover}
   $\phi\bigl(\{0,1,2\}^\FF\bigr)=X_f$.
\end{lemma}

\begin{proof}
   Let $x\in X_f$. By Lemma \ref{lem:4-cover}, there is a $d\in\{0,1,2,3\}^\FF$ with $\phi(d)=x$. Using ~$d$ we inductively construct a sequence of points in $\{0,1,2,3,4,5\}^\FF$ all of which map to $x$ under $\phi$, and such that any limit point $e$ of this sequence is contained in $\{0,1,2\}^\FF$. Then $\phi(e)=x$ by continuity of $\phi$.
   
   Let $B_n=\{s\in\FF:|s|_S\le n\}$. Fix $n\ge1$. We inductively construct $d^{(n)}$, $d^{(n-1)}$, $\ldots$, $d^{(-n-1)}$ in $\ell^{\infty}(\FF,\ZZ)$ with the following properties:
   \begin{enumerate}[label=(\arabic*)]
      \item $0\le d_s^{(k)}\le 2$ if $s\in B_n$ and $[s] >k$,
      \item $0\le d_s^{(k)}\le 5$ if $s\in B_n$ and $[s] =k$,
      \item $0\le d_s^{(k)}\le 3$ if $s\in B_n$ and $[s] <k$,
      \item $0\le d_s^{(k)}\le 5$ if $s\notin B_n$,
      \item $\phi\bigl (d^{(k)} \bigr)=x$.
   \end{enumerate}
   
   The element $d^{(n)}=d$ trivially satisfies (1)--(5). Suppose we have found $d^{(k)}$ satisfying (1)--(5) for some $k$ with $-n\le k\le n$. Construct $d^{(k-1)}$ as follows. If $[s]\neq k-1 \text{ or } k$, put $d_s^{(k-1)}=d_s^{(k)}$. For each $s\in B_n$ with $[ s]=k$, if $0\le d_s^{(k)}\le 2$ put $d_s^{(k-1)}=d_s^{(k)}$, otherwise put $d_s^{(k-1)}=d_s^{(k)} - 3$ and add 1 to the coordinates at $sa^{-1}$ and at $sb^{-1}$. Let $d^{(k-1)}$ denote the result after all these operations are carried out. 
   
   We claim that $d^{(k-1)}$ satisfies (1)--(5) with $k$ replaced by $k-1$. By construction, $d_s^{(k-1)}=d_s^{(k)}$ whenever $[s]$ is not $k-1$ or $k$, verifying ~(3). If $s\in B_n$ and $[s]=k$, then $0\le d_s^{(k)} \le 5$ and so $d_s^{(k-1)}$, which is $d_s^{(k)}$ reduced by 3 if it is more than 2, satisfies $0\le d_s^{(k-1)} \le 2$, verifying ~(1). If $s\in B_n$ and $[s]=k-1$, then $0\le d_s^{(k)}\le 3$, and  $d_s^{(k-1)}$ is either $d_s^{(k)}$, $d_s^{(k)}+1$, or $d_s^{(k)}+2$, depending on the coordinates at $sa$ and at $sb$, verifying ~(2). If $s\notin B_n$, then $0\le d_s^{(n)} = d_s \le 3$. When constructing the $d^{(k)}$ the coordinate at $s$ can change at most once, when $k=[s]+1$, and in this case can increase only by 0, 1, or 2, depending on the coordinates at $sa$ and $sa$, verifying ~(4). Finally, the construction of $d^{(k-1)}$ shows that $d^{(k-1)}= d^{(k)}- g\cdot f^*$ for some $g\in\ZZ\FF$. Hence 
   \begin{displaymath}
      \phi\bigl( d^{(k-1)} \bigr) =\pi \bigl( d^{(k)}\cdot \wdel - g\cdot f^*\cdot\wdel\bigr)
      = \phi\bigl( d^{(k)}\bigr) -\pi(g)  =\phi\bigl( d^{(k)} \bigr),
   \end{displaymath}
   verifying (5) by induction.

   By compactness of $\{0,1,...,5\}^{\FF}$, the sequence $\{d^{(-n-1)}\}$ has a convergent subsequence, say with limit $e$. Then $e\in\{0,1,2\}^{\FF}$ by (1), and $\phi(e)=x$ by (5) and continuity of $\phi$.
\end{proof}

\section{Injectivity of the homoclinic map}\label{sec:percolation}

Here we show that the homoclinic map $\phi\colon Y\to X_f$ is one-to-one off a $\nu$-null
subset of $Y$. The proof uses a modification of the percolation argument in \cite{EinsiedlerSchmidt}*{Prop.\ 5.1}.

\begin{proposition}\label{prop:one-to-one}
   Let $\phi\colon (Y,\sigma,\nu)\to (\xf,\af,\mf)$ be the homoclinic map. Then there is a $\sigma$-invariant subset $E\subset Y$ with $\nu(E)=0$ and such that $\phi$ is one-to-one on $Y\setminus E$. 
\end{proposition}

\begin{proof}
   Let $d\in Y$. Suppose that there is an $e\in Y$ such that $e\ne d$ and $\phi(e)=\phi(d)$. Then $(e-d)\cdot \xdel=0$, so that $c := (e-d)\cdot\wdel\in\lifz$. Since $-2\le e-d\le2$ and $\wdel$ is a probability distribution, it follows that $-2\le c\le 2$. Furthermore, $c \cdot f^*=(e-d)\cdot\wdel\cdot f^*=e-d$, and so $-2\le c\cdot f^*\le2$. This condition defines a shift of finite type $\Sig\subset\{-2,\dots,2\}^{\FF}$ consisting of all $c$ for which $-2\le 3c_s-c_{sa}-c_{sb}\le2$ for every $s \in\FF$. A direct calculation shows that there are 41 triples $(k,l,m)\in\{-2,\dots,2\}^3$ with $-2\le 3k-l-m\le2$, and these are the allowed patterns for $(c_s,c_{sa},c_{sb})$ that define $\Sig$.

   First suppose that $c\in\Sig$ has $c_s=2$ for some $s\in\FF$. The only allowed pattern of the form $(2,l,m)$ is $(2,2,2)$, showing that $c_{sa}=2$ and $c_{sb}=2$ as well. Repeating this argument shows that $(c_{sa^n},c_{sa^{n+1}},c_{sa^nb})=(2,2,2)$ for every $n\ge0$. Hence $(c\cdot\fs)_{sa^n}=2$ for every $n\ge0$. Since $0\le e,d\le2$ and $e_{sa^n}-d_{sa^n}=(c\cdot\fs)_{sa^n}=2$, we conclude that $d_{sa^n}=0$ for every $n\ge0$. Hence $d$ is contained in a $\nu$-null set $E^{(s)}$. Letting $E_1=\bigcup_{s\in\FF}E^{(s)}$, we see that $\nu(E_1)=0$ and that if $c_s=2$ for some $s\in\FF$ then $d\in E_1$.

   The case $c_s=-2$ for some $s\in\FF$ is similar, resulting in another $\nu$-null set $E_2$.

   Thus we are reduced to the case $-1\le c\le1$ and the corresponding shift of finite type $\Sig'\subset \{-1,0,1\}^{\FF}$ defined by the same finite-type condition $-2\le 3c_s-c_{sa}-c_{sb}\le 2$ for every $s\in\FF$. Another direct calculation shows that there are 15 allowed patterns in $\{-1,0,1\}^3$, a subset of the 41 patterns above.

   Let $c\in\Sig'$ with $c_s=1$ for some $s\in \FF$. The only allowed patterns of the form $(1,l,m)$ are $(1,0,1)$, $(1,1,0)$, and $(1,1,1)$. Fix $n\ge1$. From $c$ we construct a word $p=p_1p_2\dots p_n$ with $p_j=a$ or $b$ inductively as follows. Denote $p_1p_2\dots p_k$ by $s_k$ for $1\le k\le n$, and define $s_0=1$, so that $c_{ss_0}=1$. Suppose that $p_1,\dots,p_k$ have been found so that $c_{ss_k}=1$. If $c_{ss_ka}=1$, then put $p_{k+1}=a$, otherwise put $p_{k+1}=b$.

   This process guarantees that $c_{ss_k}=1$ for every $0\le k\le n$, but also provides more information about other coordinates of $c$ which we use to constrain the coordinates of $d$.

   If $p_{k+1}=b$, then $(c_{ss_k},c_{ss_ka},c_{ss_kb})=(1,0,1)$, so that $(c\cdot\fs)_{ss_k}=2$, forcing $d_{ss_k}=0$ as before.

   If $p_{k+1}=a$, there are two cases for $(c_{ss_k},c_{ss_ka},c_{ss_kb})$, either $(1,1,0)$ or $(1,1,1)$. In the first case, $(c\cdot\fs)_{ss_k}=2$, and again $d_{ss_k}=0$. In the second case, $(c\cdot\fs)_{ss_k}=1$, and so $(e_{ss_k},d_{ss_k})$ is either $(1,0)$ or $(2,1)$. But this also means that $c_{ss_k b}=1$, and so $e_{ss_kb}-d_{ss_kb}=(c\cdot\fs)_{ss_kb}=1$ or $2$, and in either case $d_{ss_kb}=0$ or $1$. Hence $(d_{ss_k},d_{ss_kb})$ can be only one of five out of nine possible pairs, namely $(0,0)$, $(0,1)$, $(0,2)$, $(1,0)$, or $(1,1)$. Observe that since $[ss_j]=[s]+j$ and $b\neq a$, it follows that $ss_kb$ cannot occur among the $ss_j$ for $0\le j\le n$.

   Let $m$ be the number of $a$'s appearing in $p=p_1p_2\dots p_n$. Then $d$ is contained in a subset of $Y$ of measure $(5/9)^m(1/3)^{n-m}$, one factor of $5/9$ for each $a$ in $p$ and one factor of $1/3$ for each $b$. Thus summing over all possible words $p$, we see that any $d$ in this case must lie in a set $E^{(s,n)}\subset Y$ with
   \begin{equation}\label{eqn:injective}
      \nu(E^{(s,n)})\le \sum_{m=0}^n \binom{n}{m}\Bigl(\frac{5}{9}\Bigr)^m\Bigl(\frac13\Bigr)^{n-m} =
      \Bigl(\frac89\Bigr) ^n\to 0 \text{\quad as $n\to\infty$}.
   \end{equation}
   Since $E^{(s,n+1)}\subset E^{(s,n)}$ for every $n\ge1$, their intersection $E^{(s)} :=\bigcap_{n=1}^\infty E^{(s,n)}$ has $\nu(E^{(s)})=0$. Hence $E_3:=\bigcup _{s\in\FF}E^{(s)}$ is also $\nu$-null. This shows that if $e-d=c\cdot\fs$ and $c_s=1$ for some $s\in\FF$, then $d\in E_3$. The case $c_s=-1$ is exactly the same, resulting in a further $\nu$-null set $E_4$.

   Thus if $d$ is not in the $\sigma$-invariant $\nu$-null set $\bigcup_{j=1}^4 E_j$, then there is no $e\neq d$ with $\phi(e)=\phi(d)$, concluding the proof.
\end{proof}

We remark that although $\phi$ is one-to-one $\nu$-almost everywhere, there are subsets of $Y$ of large cardinality that map to a common point. For example, the equation $\phi(d)=0$ leads to an uncountable shift of finite type in $\{0,1,2\}^\F$ whose image under the map $c\mapsto c\cdot f^*$ is the set of solutions.

\section{Isomorphism}\label{sec:isomorphism}

To show that the homoclinic map $\phi\colon (Y,\sigma,\nu)\to(\xf,\af,\mf)$ is a measurable isomorphism of $\FF$-actions, it only remains to show that $\phs\nu=\mf$. The sofic entropy of $\af$ with respect to $\mf$ equals $\log 3$ by \cite{BowenExpansive}*{Thm.\ 1.2}. Furthermore, the sofic entropy of $\af$ with respect to $\phs\nu$ is also $\log 3$ by \cite{BowenNew}*{Prop.\ 2.2} using the homoclinic isomorphism with the full 3-shift. If we knew that Haar measure were the unique $\af$-invariant measure of maximal entropy, we would be done. This is indeed the case for expansive algebraic actions of amenable groups with completely positive entropy \cite{ChungLi}*{Thm.\ 8.6}, but remains open for actions of general sofic groups, and in particular for free groups. Bowen \cite{Bowen}*{Thm.\ 8.2} has constructed a cautionary example of a transitive shift of finite type over $\FF$ with (at least) two measures of maximal entropy.

Thus a different proof that $\phs\nu=\mf$ is necessary. Our proof creates enough group-like structure in $Y$ to show that $\phs\nu$ is invariant under translation by every element in the homoclinic group $\Df$. Then density of $\Df$ in $\xf$ implies that $\phs\nu$ is a translation-invariant probability measure on $\xf$, and hence must coincide with Haar measure~ $\mf$.

\begin{proposition}\label{prop:addition}
   Let $\del\in Y$ be the element given by $\del_{1_\FF}=1$ and $\del_s=0$ for every $s\neq 1_\FF$. Define a map $\tau\colon Y\to Y$ by $\tau(d)=\rho(d+\del)$, where $\rho$ denotes the reduction process from the proof of Lemma \ref{lem:3-cover}. Then $\tau$ is well-defined and one-to-one off a $\nu$-null set,  $\tau_*\nu=\nu$, and $\phi(\tau(d))=\phi(d)$ for every $d\in Y$.
\end{proposition}

\begin{proof}
   We decompose $Y$ into a countable collection of disjoint cylinder sets whose union has full measure, and such that $\tau$ has the required properties on each cylinder set. To do this, we introduce a tree structure that reflects coordinates affected by the reduction process applied to $d+\del$.

   For notational simplicity, let $A=a^{-1}$ and $B=b^{-1}$. Denote the set of all words in $A$ and $B$ (including $1$) by $N$. If $s=s_1s_s\dots s_n\in N$, an \emph{initial subword} of $s$ is one of the form $s_1s_2\dots s_k$ for some $0\le k\le n$, where by convention this product is $1$ if $k=0$. A \emph{tree} is a finite subset of $N$ that is closed under taking initial subwords. If $T$ is a tree, we define $\tb=T\cup TA \cup TB$, where $TA=\{tA:t\in T\}$ and $TB=\{tB:t\in T\}$, and $\pt=\tb\smallsetminus T$. For example, if $T=\{1,A,B,AB,BA,BB\}$, then $\tb=\{1,A,AA,AB,ABA,ABB,B,BA,BAA,BAB,BBA,BBB\}$ and $\pt=\{AA,ABA,ABB,BAA,BAB,BBA,BBB\}$.

   Each tree corresponds to an ordered binary tree, and conversely. From basic and well-known properties of such trees we know that $|\tb|=2|T|+1$ and that $|\pt|=|T|+1$.

   Let $T$ be a tree and $\om\in\{0,1\}^{\pt}$. Define the cylinder set $\xxx \subset Y$ by $\xxx $ $\fhat$
   \begin{displaymath}
      \xxx=\{d\in Y: d_t=2 \text{\quad for every $t\in T$ and $d_s=\om_s$ for every $s\in\pt$}\}.
   \end{displaymath}
   By convention, we allow $T=\emptyset$ and in this case define $\tb=\{1_\G\}$ and $\pt=\{1_\G\}$. Observe that if $d\in\xxx$, then the reduction process resulting in $\rho(d+\del)$ will halt after finitely many steps, alter only the coordinates of $d$ within $\tb$, and have value
   \begin{displaymath}
      \rho(d+\del)_s=
      \begin{cases}
         0         &  \text{if $s\in T$}, \\
         d_s+1     &  \text{if $s\in\pt$}, \\
         d_s       &  \text{if $s\notin\tb$}.
      \end{cases}
   \end{displaymath}
   Clearly $\tau$ is one-to-one on $\xxx$, and $ \nu\bigl(\tau(\xxx)\bigr)=\nu(\xxx)=(1/3)^{|\tb|}$.

   The collection $\{\xxx:\text{$T$ is a tree and\ } \om\in\{0,1\}^{\pt} \}$ is pairwise disjoint, and the images of these sets under $\tau$ are also pairwise disjoint. It is known that
   \begin{displaymath}
      |\{T:|T|=n\}|= C_n=\frac{1}{n+1}\binom{2n}{n},
   \end{displaymath}
   where $C_n$ is the $n$th Catalan number, and that $\displaystyle \sum_{n=0}^\infty C_nu^n= \frac{2} {1+\sqrt{1-4u}}$. Hence
   \begin{displaymath}
     \sum_{T}\! \sum_{\om\in\{0,1\}^{\pt}} \!\nu(\xxx) = \sum_T 2^{|\pt|}\Bigl(\frac13\Bigr)^{|\tb|}
     =\sum_{n=0}^\infty C_n 2^{n+1}\Bigl(\frac13\Bigr)^{2n+1} = \frac23 \sum_{n=0}^\infty C_n
     \Bigl(\frac29\Bigr)^n = 1,
   \end{displaymath}
   proving that $\tau$ is well-defined and one-to-one off a $\nu$-null set, and that $\tau_*\nu=\nu$. The reduction process $\rho$ does not affect the image under $\phi$, and so $\phi(\tau(d))=\phi(\rho(d+\delta))=\phi(d+\delta)=\phi(d)+\xdel$.
\end{proof} 

\begin{proof}[Proof of Theorem \ref{thm:main}]
   If $\tau\colon Y\to Y$ is the map defined in Proposition \ref{prop:addition}, then for every $s\in\FF$ we have that $(\sig_s\circ \tau\circ\sig_s^{-1})(d)=\rho(d+s\cdot\del)$. Since both $\sig_s$ and $\tau$ preserve ~$\nu$, and since
   \begin{displaymath}
      \phi\bigl( (\sig_s^{}\circ\tau\circ\sig_s^{-1})(d)\bigr)=\phi\bigl(\rho(d+s\cdot\del)
      \bigr) = \phi(d+s\cdot\del)=\phi(d)+s\cdot\xdel,
   \end{displaymath}
   it follows that for every $s\in\FF$ and every Borel set $E\subset X_f $ we have that
   \begin{align*}
      (\phi_*\nu)(E + s\cdot \xdel) &= \nu(\phi^{-1}(E+s\cdot\xdel)) =
      \nu((\sigma_s\circ\tau\circ\sigma_s^{-1})(\phi^{-1}(E)) \\
       &= \nu(\phi^{-1}E) =
      (\phi_*\nu)(E).
   \end{align*}
   Hence $\phs\nu$ is invariant under translation by all integral combinations of shifts of $\xdel$, i.e., under translation by all elements in $\Df$. Then density of $\Df$ implies that $\phs\nu$ is translation-invariant, and so $\phs\nu=\mf$.
\end{proof}

\section{Generalizations}\label{sec:generalizations}

In this section we generalize Theorem~\ref{thm:main} to a larger class of acting groups and also to more principal actions.

A countable group $\G$ is called \emph{indicable} if there is a surjective homomorphism $[\,\cdot\,]\colon \G \to \ZZ$. Many of our earlier arguments extend to indicable acting groups. An exception is the combinatorial proof of Proposition \ref{prop:addition}, from which we deduced that $\phs\nu=\mu_f$. However, we can substitute an analytic alternative due to Hayes \cite{HayesMax}*{Thm.\ 3.6} which is both more general and in addition establishes the surjectivity of the homoclinic map without resorting to the reduction process in Lemma \ref{lem:3-cover}.

We can also adapt our arguments to elements of the form $M-a-b\in\ZZ\G$ where $M\ge3$. The main issue here is extending Proposition \ref{prop:one-to-one} to prove injectivity. 
One might expect that the shift of finite type within $\{-M+1,\dots,M-1\}^\G$ analogous to $\Sigma'\subset \{-1,0,1\}^\F$ would be significantly more complicated. However, it turns out that these coincide for all $M\ge3$, and so the argument can be applied almost verbatim.

\begin{theorem}\label{thm:generalization}
   Let $\G$ be a countable group equipped with a homomorphism $[\,\cdot\,]\colon \G\to\ZZ$. Suppose that $a$ and $b$ are distinct elements of $\G$ with $[a]=[b]=1$, and that $M\ge3$ is an integer. Let $f = M-a-b\in\ZZ\G$. Then the principal algebraic $\G$-action $(X_f,\af)$ is measurably isomorphic to the Bernoulli $\G$-action on $\{0,1,\dots,M-1\}^\G$ with the uniform base probability measure.
\end{theorem}

We start by describing some routine extensions needed. As before, put $f^*=M-a^{-1}-b^{-1}$, and let $\wdel=(f^*)^{-1}\in\ell^1(\G,\RR)$. Then $\wdel_s\ge 0$ for every $s\in\G$, and $\|\wdel\|_1=1/(M-2)$. Let $\xdel=\pi(\wdel)\in\TT^{\G}$.

Let $Y_0=\{0,1,\dots,M-1\}$ and $\nu_0$ be the uniform probability measure on $Y_0$. The Fourier transform of $\nu_0$ is given by $\widehat{\nu}_0(\xi)=(1/M)\sum_{k=0}^{M-1} e^{2\pi i k \xi}$. Let $Y=Y_0^{\G}$ and $\nu=\nu_0^{\otimes\G}$ be product measure on $Y$. The homoclinic map $\phi\colon Y \to X_f$ defined by $\phi(d)=d\cdot \xdel = d\cdot \pi(\wdel)$ is continuous. 

Since $f^{-1}\in\ell^1(\G,\RR)$, again by geometric series, it follows that $g\cdot f^{-1} \in \ell^1(\G,\RR)$ for every $g\in\ZZ\G$. We abbreviate $g\cdot f^{-1}$ to $g/f$. We caution the reader that points of the form $\pi(g/f)$ are \emph{not} in $X_f$, but rather in $X_{f^*}$, which can be rather different. 

The following is a special case of a result by Hayes \cite{HayesMax}*{Thm. 3.6}. Our case relatively easy, and for the convenience of the reader we give a direct proof. 

\begin{proposition}\label{prop:fourier}
   With the above notations, let $\mu=\phs\nu$, considered as a measure on $\TT^{\G}$. For every $g\in \ZZ\G$ we have that
   \begin{displaymath}
      \widehat{\mu}(g)=\prod_{s\in\G} \widehat{\nu}_0\bigl( (g/f)_s\bigr),
   \end{displaymath}
   and this product is absolutely convergent.
\end{proposition}

\begin{proof}
   By definition,
   \begin{displaymath}
      \widehat{\mu}(g)=\int_{\TT^\G} e^{-2\pi i \<\xi,g\>}\,d\mu(\xi)
      =\int_Y e^{-2 \pi i \< \phi(d),g\>}\, d\nu(d).
   \end{displaymath}
   In order to determine the exponent, note that $(\wdel)^*=1/f$ and so
   \begin{align*}
   	\bigl( (d\cdot \wdel)\cdot  g^*\bigr)_{1_\G} &=\sum_{s\in\G} d_s \Bigl( 
   	\sum_{tu=s^{-1}} \wdel_t\csp g_u^* \Bigr) \\
   	&= \sum_{s\in\G} \Bigl( \sum_{u^{-1}t^{-1}=s} g_{u^{-1}}^{}(\wdel)^*_{t^{-1}}\Bigr) d_s = \sum_{s\in\G} (g/f)_s\csp d_s.
  	\end{align*}
  	Hence
  	\begin{align*}
  		\exp\bigl[ -2\pi i & \<\phi(d),g\>\bigr] = \exp\bigl[ -2\pi i \bigl((d\cdot\wdel)\cdot g^*\bigr)_{1_\G} \bigr]\\
  		&= \exp\bigl[-2\pi i\sum_{s\in\G} (g/f)_s\,d_s\bigr] =
  		\prod_{s\in\G} \exp\bigl[-2\pi i (g/f)_s\,d_s\bigr] .
  	\end{align*}
  	Thus
  	\begin{align*}
  	  \widehat{\mu}(g) & = \int_Y \prod_{s\in\G} e^{-2\pi i (g/f)_s d_s}\,d\nu_0^{\otimes\G} (d)\\
  	  & = \prod_{s\in\G} \int_{Y_0} e^{-2\pi i (g/f)_s d_s} \,d\nu_0(d_s)
  	  =\prod_{s\in\G} \widehat{\nu}_0\bigl( (g/f)_s\bigr).
  	\end{align*}
  	Since $\nuhat_0$ is smooth with $\widehat{\nu}_0(0)=1$, and since  $g/f\in\ell^1(\G,\RR)$, the last product is clearly absolutely convergent.
\end{proof}

The preceding result is valid in great generality, for example for all polynomials with a summable inverse. However, for our purposes we need more information about the coordinates of $g/f$.

\begin{lemma}\label{lem:rational}
   Under the hypotheses of Theorem \ref{thm:generalization}, for every $g\in\ZZ\G\setminus \ZZ\G f$ there is an $s\in\G$ such that $\pi\bigl( (g/f)_s\bigr) = k/M$ for some $1\le k\le M-1$.
\end{lemma}

\begin{proof}
   If $g\notin \ZZ\G f$ then $\pi(g/f)\neq 0$ in $\TT^{\G}$. Since $(1/f)_s=0$ for every $s\in\G$ with $[s]<0$, it follows that $\{[s]:\pi((g/f))_s\neq0\}$ is bounded below. Choose $s_0$ that attains this minimum. Since $\pi(g/f)\cdot f=\pi((g/f)\cdot f) =\pi (g)=0$, we obtain that
   \begin{displaymath}
      M \pi(g/f)_{s_0} - \pi(g/f)_{s_0a^{-1}} - \pi(g/f)_{s_0b^{-1}} = 0.
   \end{displaymath}
   But the second and third terms vanish by minimality of $[s_0]$, showing that $\pi(g/f)_{s_0}$ has the required form.
\end{proof}

\begin{proof}[Proof of Theorem \ref{thm:generalization}]
   First observe that $\nuhat_0^{}(k/M)=0$ for $1\le k\le M-1$, while $\nuhat_0|_{\ZZ}\equiv 1$.  Furthermore, $\muhat(g) = \prod_{s\in\G} \nuhat_0\bigl( (g/f)_s\bigr)$ for every $g\in\ZG$ by Proposition ~\ref{prop:fourier} and Lemma~ \ref{lem:rational}. If $g\in\ZG f$, then $g/f$ has integral coordinates and hence $\muhat(g)=1$. If $g\in\ZG\setminus\ZG f$, then there is an $s$ for which $\pi((g/f)_s)=k/M$ for some $1\le k\le M-1$. Then $\nuhat_0((g/f)_s)=0$ and so $\muhat(g)=0$. Hence $\muhat$ and $\muhat_f$ both equal the indicator function of $\ZG f$, and so $\mu=\mu_f$.
   
   Since $\phi$ is continuous and $\phs\nu=\mu_f$ has full support, it follows that $\phi$ is surjective.
   
   Finally, consider the proof for injectivity in Proposition \ref{prop:one-to-one}. Suppose that $e,d\in Y$ with $\phi(e)=\phi(d)$. Then $e\cdot \wdel- d\cdot\wdel=c\in\ell^{\infty}(\G,\ZZ)$. But
   \begin{displaymath}
      \|e\cdot\wdel-d\cdot\wdel\|_{\infty} \le \|e-d\|_{\infty}\|\wdel\|_1
      \le \frac{M-1}{M-2}.
   \end{displaymath}
   If $M\ge4$, then $(M-1)/(M-2)<2$, and so $-1\le c\le 1$. Furthermore, the allowed blocks in defining the shift of finite type $\Sigma$ over $\FF$ are exactly those in the proof that remain after the possibility that $c_s=2$ or $c_s=-2$ have been dealt with. Explicitly, $\{(k,l,m)\in\{-1,0,1\}^3: -M+1\le 3k-l-m\le M-1\}$ is the same set of 15 patterns. The proof now proceeds as before, with the bound of $(8/9)^n$ replaced by $[(2M+2)/M^2]^n$.
\end{proof}

\section{Remarks and questions}\label{sec:remarks}

Suppose that $\G$ is a countable group and that $[\,\cdot\,]\colon\G\to\ZZ$ is a homomorphism. Let $f\in\ZG$ have the form $f=M-\sum_{s\in I} f_s\csp s$, where both $[s]\ge1$ and $f_s>0$ hold for every $s\in I$, and where $M>\sum_{s\in I}f_s$. Using the same notation as in \S\ref{sec:generalizations}, we see that $\wdel=(f^*)^{-1}\in\ell^1(\G,\RR)$ defines the continuous homoclinic map $\phi\colon Y\to X_f$. Proposition ~\ref{prop:fourier} and Lemma ~\ref{lem:rational} easily extend to show that $\phi$ is surjective and maps the Bernoulli measure $\nuhat_0^{\otimes \G}$ to Haar measure ~$\mu_f$. What is not obvious is whether $\phi$ is essentially one-to-one, although this seems to us likely.

\begin{conjecture}\label{conj:bernoulli}
   Let $\G$ be a countable group and $[\,\cdot\,]\colon\G\to\ZZ$ be a homomorphism. Suppose that $f=M-\sum_{s\in I} f_s\csp s$, where $M>\sum_{s\in I}f_s$ and  $[s]\ge1$ and $f_s>0$ for every $s\in I$.  Then the homoclinic map $\phi\colon Y\to X_f$ is a measurable isomorphism.
\end{conjecture}

The condition that the large coefficient of $f$ occur at an extreme coordinate with respect to $[\,\cdot\,]$ is certainly necessary. Take for example $\G=\ZZ=\< a\>$ and $f = 3 - a - a^{-1}$. Then $\af$ has entropy $\log [(3+\sqrt{5})/2] <\log 3$. Here Lemma \ref{lem:rational} breaks down, since the coordinates of $1/f$ are irrational. Indeed, it is even possible for these coordinates to be transcendental \cite{LindSchmidtHomoclinic}*{Example 5.8}. Einsiedler and the second author constructed an explicit sofic shift and a continuous map from it to ~$X_f$ that is essentially one-to-one \cite{EinsiedlerSchmidt}*{Example 4.1}. The extent to which algebraic actions have such ``good'' symbolic covers has been extensively studied for $\G=\ZZ^d$, and now presents new possibilities for general ~$\G$.

As observed in \cite{EinsiedlerSchmidt}*{Cor.\ 5.1} in the case $\G=\ZZ^2$, by varying $[\,\cdot\,]$ we can conclude that each of the four elements $3-a^{\pm1}-b^{\pm1}$ give an algebraic $\G$-action isomorphic to the same Bernoulli $\G$-action, and hence are isomorphic to each other. However, changing the coefficient signs can seriously impede our analysis. For instance, using the notation in Conjecture \ref{conj:bernoulli}, does $f = 3 + a + b$ define an algebraic $\G$-action that is Bernoulli? 

An element $f\in\ZG$ is \emph{lopsided} if there is an $s_0\in\G$ such that $|f_{s_0}| > \sum_{s\neq s_0}|f_s|$. For an arbitrary countable group $\G$, is every principal algebraic $\G$-action defined by a lopsided polynomial measurably isomorphic to a Bernoulli $\G$-action? Hayes ~\cite{HayesHarmonic} showed that every such action is a factor of a Bernoulli action under some mild orderability assumptions on $\G$. It follows that if $\G$ is amenable, then the action itself is Bernoulli by the results of Ornstein and Weiss \cite{Ornstein-Weiss}. However, to our knowledge this remains open for nonamenable groups, and even for free groups.

Next, consider the case $\G=\F$ and $f=2-a-b$. If we attempt to mimic earlier constructions, we immediately hit a roadblock that although $1/f^*$ is well-defined, it is no longer in $\ell^1(\F,\RR)$, and so the convolution operator used to define the homoclinic map has no clear meaning. However, here $1/f^* \in\ell^2(\F,\RR)$, and recent work of Hayes \cite{HayesMax} shows that convolution can be extended to square summable elements, using convergence in measure. As a consequence, he obtains a well-defined measurable homoclinic map $\phi\colon\{0,1\}^\F \to X_f$, and his more general version of Proposition \ref{prop:fourier}  shows that $\phi$ maps the Bernoulli measure to Haar measure. However, this leaves open the enticing problem of whether this explicit map is essentially one-to-one.

Finally, let $\G=\F$ and $f = 1+a+b$. Historically, it was the careful computation of the entropy of the commutative version of this example that was the key to unlocking the connection between entropy for algebraic actions and Mahler measure ~\cite{LSW}. There are tantalizing clues that the sofic entropy of $\af$ is not only positive, but has the precise value $\log(3/\sqrt{2})$. Essentially nothing is known about the dynamical properties of $\af$. Is it mixing? Does it have completely positive entropy (with respect to every sofic approximation to $\F$)? Is it measurably isomorphic to a Bernoulli $\F$-action?

\end{document}